\documentclass[makeidx]{article}
\usepackage[utf8]{inputenc}
\usepackage[T1]{fontenc}

\usepackage{graphicx}
\usepackage{psfrag}

\usepackage{makeidx}
\usepackage{amssymb , amsmath, amsthm}
\newtheorem{defi}{Definition}
\newtheorem{theorem}[defi]{Theorem}

\newtheorem{remark}[defi]{Remark}
 
\newtheorem{lemma}[defi]{Lemma}
\newtheorem{cor}[defi]{Corollary}

\usepackage{color}

\title{The spectral gap of graphs and Steklov eigenvalues on
  surfaces}
\author{Bruno Colbois and Alexandre Girouard}
\begin{document}
\maketitle

\begin{abstract}
  Using expander graphs, we construct a sequence
  $\{\Omega_N\}_{N\in\mathbb{N}}$ of smooth compact surfaces with boundary of
  perimeter $N$, and with the first non-zero Steklov
  eigenvalue $\sigma_1(\Omega_N)$ uniformly bounded away from
  zero. This answers a question which was raised in~\cite{gp3}. The
  genus of $\Omega_N$ grows linearly with $N$, this is 
  the optimal growth rate.
\end{abstract}

\section{Introduction}
Let $\Omega$ be a compact, connected, orientable smooth Riemannian
surface with boundary $\Sigma=\partial\Omega$. The Steklov eigenvalue
problem on $\Omega$ is
\begin{gather*}
  \Delta f=0 \ \text{in }\Omega,\ \ \partial_\nu f=\sigma f\ \text{on} \ \Sigma,
\end{gather*}
where $\Delta$ is the Laplace--Beltrami operator on $\Omega$ and
$\partial_\nu$ denotes the outward normal derivative along the boundary $\Sigma$.
The Steklov spectrum of $\Omega$ is denoted by
$$
0=\sigma_0 < \sigma_1(\Omega) \le \sigma_2(\Omega) \le... \nearrow \infty,
$$
where each eigenvalue is repeated according to its multiplicity.
In~\cite{gp3}, the second author and I. Polterovich asked the
following question:

\emph{Is there a sequence $\{\Omega_N\}$ of surfaces with boundary 
such that $\sigma_1(\Omega_N)L(\partial\Omega_N) \rightarrow \infty$ as $N\to\infty$?}

\medskip
The goal of this paper is to give a positive answer to this question.
\begin{theorem}\label{Theorem:MainIntro}
  There exist a sequence $\{\Omega_N\}_{N\in\mathbb{N}}$ of compact surfaces
  with boundary and a constant $C>0$ such that for each $N\in\mathbb{N}$,
  $\mbox{genus}(\Omega_N)=1+N,$ and 
%  \begin{gather}
    $$\sigma_1(\Omega_N)L(\partial\Omega_N)\geq CN.$$%\label{Limit:MainTheorem}
%  \end{gather}
\end{theorem}
For each $\gamma\in\mathbb{N}$, consider the class
$\mathcal{S}_\gamma$ of all smooth compact surfaces of genus $\gamma$
with non-empty boundary, and define
\begin{gather*}
  \sigma^\star(\gamma)=\sup_{\Omega\in\mathcal{S}_\gamma}\sigma_1(\Omega)L(\partial\Omega).
\end{gather*}
Few results on $\sigma^\star(\gamma)$ are known.
G. Kokarev proved~\cite{kokarev1} that for each genus $\gamma$,
\begin{gather}\label{BoundKokarev}
  \sigma^\star(\gamma)\leq 8\pi(\gamma+1).
\end{gather}
See also~\cite{ceg2} for a similar bound involving the higher Steklov
eigenvalues $\sigma_k(\Omega)$, and~\cite{gp3} for a bound involving
the number of boundary components.
In~\cite{fraschoen2}, A. Fraser and R. Schoen
proved that $\sigma^\star(0)=4\pi$. In this case, the supremum is attained in the limit
by a sequence of surfaces with their number of boundary components
tending to infinity. It follows from~(\ref{BoundKokarev}) that the growth of
$\sigma^\star(\gamma)$ is sublinear. % It follows from Theorem~\ref{Theorem:MainIntro} that
% it is in fact linear.
\begin{cor}\label{Theorem:genus}
  There exists a constant $C>0$ such that for each 
  $\gamma\in\mathbb{N}$,
  $\sigma^\star(\gamma)\geq C\gamma.$
\end{cor}
In the construction of $\{\Omega_N\}$ that we propose, the number of
boundary components also tends to infinity. It would be interesting to
know if this condition is necessary.

% \smallskip
% \noindent
% \textbf{Question.}
%   Does there exist a sequence $\{\Omega_N\}$ of compact surfaces with boundary
%   which satisfies (\ref{Limit:MainTheorem}) and has uniformly bounded
%   number of boundary components?

\begin{remark}
  The problem of constructing closed surfaces $M$ with large normalized
  first non-zero eigenvalue $\lambda_1(M)\mbox{Area}(M)$ has been considered by several
  authors. See for instance~\cite{buser1,brooks1,brooks2}. Our proofs
  uses methods which are related to those of~\cite{colbmat}.
\end{remark}
\subsubsection*{Plan of the paper}
In Section~\ref{Section:construction}, we present the construction of
a surface $\Omega_\Gamma$ which is obtained from a regular graph
$\Gamma=(V,E)$ by sewing copies of a \emph{fundamental piece}
following the pattern of the graph $\Gamma$. In
Section~\ref{Section:comparison}, we introduce the spectrum
of the graph $\Gamma$ and state a comparison result
(Theorem~\ref{mainthm}) between $\lambda_1(\Gamma)$ and
$\sigma_1(\Omega_\Gamma)$. This is then used, in conjunction with
expander graphs, to prove Theorem~\ref{Theorem:MainIntro}. In
Section~\ref{Section:ProofComparison}, we present the comparison
argument leading to the proof of Theorem~\ref{mainthm}.

\begin{remark}
  While this paper was in the final stage of its preparation, we
  learned that Mikhail Karpukhin also has developed a method for
  construction surfaces with large normalized Steklov eigenvalue
  $\sigma_1L$. His approach is different, and his work will appear
  in~\cite{Karpukhin}.
\end{remark}

\section{Constructing manifolds from graphs}\label{Section:construction}
Let $\Gamma$ be a finite connected regular graph of degree $k$. The
set of vertices of $\Gamma$ is denoted $V=V(\Gamma)$, the set of edges
is denoted $E=E(\Gamma)$. The number of vertices of $\Gamma$ is
$|V(\Gamma)|$. We will construct a Riemannian surface
$\Omega_\Gamma$ modelled on the graphs $\Gamma$ 
from a fixed orientable Riemannian surface $M_0$ which we call
the \emph{fundamental piece} (See Figure~\ref{fundamentalpiece}) and
which is assumed to satisfy the following hypotheses:
\begin{enumerate}
\item The boundary of $M_0$ has $k+1$ components $\Sigma_0$,
  $B_1,\cdots,B_{k}$.
\item Each of the boundary component is a geodesic curve of length 1.
\item The component $\Sigma_0$ has a neighbourhood
  which is isometric to the cylinder $C_0=\Sigma_0 \times[0,1]\subset M_0$,
  and $\Sigma_0$ corresponds to $\{0\}\times\Sigma$.
\end{enumerate}
\begin{figure}
  \centering
  \psfrag{b1}[][][1]{$B_1$}
  \psfrag{b2}[][][1]{$B_2$}
  \psfrag{b3}[][][1]{$B_3$}
  \psfrag{b4}[][][1]{$B_4$}
  \psfrag{S0}[][][1]{$\Sigma_0$}
  \psfrag{C0}[][][1]{$C_0$}
  \includegraphics[width=5cm]{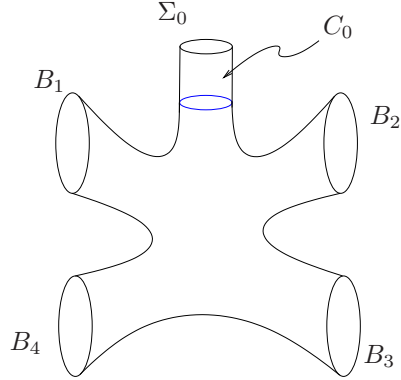}
  \caption{The fundamental piece $M_0$ for a graph $\Gamma$ of degree 4.}
  \label{fundamentalpiece}
\end{figure}
The manifold $\Omega_\Gamma$ is obtained by sewing copies of the
fundamental piece $M_0$ following the pattern of the graph
$\Gamma$: to each vertex $v\in V$, there corresponds a isometric
copy $M_v$ of $M_0$. The edges emanating from a vertex $v\in
V(\Gamma)$ are labelled $e_1(v),\cdots,e_k(v)$.
The corresponding boundary components
$B_1,\cdots,B_{k}$ are identified along these edges: if $v\sim w$ then
there are $1\leq i,j\leq k$ such that $e_i(v)=e_j(w)$ and the boundary
component $B_i$ of $M_v$ is identified to the boundary component $B_j$
of $M_w$. The manifold $\Omega_\Gamma$ has one boundary 
component $\Sigma_v$ for each vertex $v\in V(\Gamma)$, each of them
being isometric to $\Sigma_0$ with corresponding cylindrical
neighbourhood $M_v\supset C_v\cong [0,1]\times\Sigma_0$

The following lemma shows that the genus of $\Omega_\Gamma$ grows
linearly with the number of vertices of the graph $\Gamma$.
\begin{lemma}\label{Lemma:genus}
  The genus of the surface $\Omega_\Gamma$ is
  \begin{gather*}
    \gamma(\Omega_\Gamma)=1+\left(\gamma(M_0)+\frac{k}{2}-1\right)|V(\Gamma)|,
  \end{gather*}
  where $\gamma(M_0)$ is the genus of the fundamental piece $M_0$, and
  $|V(\Gamma)|$ is the number of vertices of $\Gamma$.
\end{lemma}
\begin{remark}
  Because the number of vertices of odd degree is always even,
  $1+(\gamma(M_0)+\frac{k}{2}-1)|V(\Gamma)|$ is an integer.
\end{remark}  
\begin{proof}[Proof of Lemma~\ref{Lemma:genus}]
  The genus $\gamma$ and the Euler--Poincaré characteristic $\chi$ of a
  smooth compact orientable surface with $b$ boundary components are related by
  the formula
  $$\chi=2-2\gamma-b.$$
  Let $K:\Omega_\Gamma\rightarrow\mathbb{R}$ be the Gauss
  curvature. Since the boundary curves $\Sigma_0,B_1,\cdots,B_k$ are geodesics,
  it follows  from the Gauss--Bonnet formula hat
  \begin{align*}
    \chi(\Omega_\Gamma)=\frac{1}{2\pi}\int_{\Omega_\Gamma}K
    =\frac{1}{2\pi}|V(\Gamma)|\int_{M_0}K
    =\chi(M_0)|V(\Gamma)|
    =(1-2\gamma(M_0)-k)|V|,
  \end{align*}
  where we have used that the number of boundary components of $M_0$
  is $k+1$.
  It follows that the genus of $\Omega_\Gamma$ is
  \begin{gather*}
    \gamma(\Omega_\Gamma)=\frac{1}{2}(2-\chi(\Omega_\Gamma)-|V|)
    =1+\left(\gamma(M_0)+\frac{k}{2}-1\right)|V|
  \end{gather*}
\end{proof}

\section{Comparing eigenvalues on graphs to Steklov eigenvalues}
\label{Section:comparison}
Our main reference for spectral theory on graphs is~\cite{Chung}.
The space $\ell^2(V(\Gamma))=\{x:V(\Gamma)\rightarrow\mathbb{R}\}$
is equipped with the norm defined by
$$\|x\|^2=\sum_{v\in V(\Gamma)} x(v)^2.$$
The discrete Laplacian $\Delta_\Gamma$ acts on
$\ell^2(V(\Gamma))$ and is defined  by the quadratic form
\begin{gather}\label{quadraticform}
  q_\Gamma(x)=\sum_{v\sim w}\bigl(x(v)-x(w)\bigr)^2,
\end{gather}
where the symbol $v\sim w$ means that the vertices $v$ and $w$ of
$\Gamma$ are adjacent, and the sum appearing in~\eqref{quadraticform}
is over all non-oriented edges of $\Gamma$.
The discrete Laplacian $\Delta_\Gamma$ has a finite non-negative
spectrum which we denote by
$$
0=\lambda_0<\lambda_1(\Gamma) \le \lambda_2(\Gamma) \le...\le
\lambda_{\vert V\vert-1}(\Gamma),
$$
where each eigenvalue is repeated according to its multiplicity.
The first non zero eigenvalue admits the following variational characterization:
\begin{gather}\label{charvargraph}
  \lambda_1(\Gamma)=
  \min\left\{
    \frac{q_\Gamma(x)}{\|x\|^2}\,\Bigl|\Bigr.\,x:V(\Gamma)\rightarrow\mathbb{R}, \sum_{v\in V(\Gamma)} x(v)=0\right\}.
\end{gather}
In order to compare $\lambda_1(\Gamma)$ to the fist non-zero
Steklov eigenvalue of $\Omega_\Gamma$, the following variational
characterization will be used:
\begin{gather*} %\label{charvar}
  \sigma_1(\Omega_\Gamma)= \inf\left\{ \int_{\Omega_\Gamma} |\nabla f|^2\,\Bigl|\Bigr.\, f\in C^\infty(\Omega_\Gamma),\,
    \int_{\partial\Omega_\Gamma}f=0,\, \int_{\partial\Omega_\Gamma}f^2=1\right\}.
\end{gather*}
The main result of this paper will follow from the following estimate.
\begin{theorem}\label{mainthm}
  There exist constants $\alpha,\beta>0$ depending only on the
  fundamental piece $M_0$ such that 
  \begin{gather*}
    \alpha\leq\frac{\sigma_1(\Omega_\Gamma)}{\lambda_1(\Gamma)}
    \leq
    \beta
  \end{gather*}
\end{theorem}
The proof of Theorem~\ref{mainthm} will be presented in Section~\ref{Section:ProofComparison}.

\subsection{Expander graphs and the proof of Theorem~\ref{Theorem:MainIntro}}
\label{subsection:proofMainTheoremIntro}
To prove Theorem~\ref{Theorem:MainIntro}, we will use expander graphs,
through one of their many characterizations. 
\begin{defi}\label{Definition:ExpanderGraphs}
  A sequence of $k$-regular graphs $\{\Gamma_N\}_{N\in\mathbb{N}}$ is called
  a \emph{family of expander graphs} if\\
  $\lim_{N\rightarrow\infty}|V(\Gamma_N)|=+\infty$ and 
  $\lambda_1(\Gamma_N)$ is uniformly bounded below by a positive
  constant.
\end{defi}
See~\cite{MR2247919} for a survey of their properties and applications. 
Consider a fundamental piece $M_0$ of genus $0$, with 5 boundary
components, that is with $k=4$. Let $\{\Gamma_N\}$ be a
family of $4$-regular expander graphs such that the number of vertices
$|V(\Gamma_N)|=N$. The existence of this family of expander graphs
follows from the classical probabilistic method~\cite{pinsker}.
It follows from Lemma~\ref{Lemma:genus} that the
genus of $\Omega_{\Gamma_N}$ is
$$\gamma(\Omega_{\Gamma_N})=1+N.$$
By definition, there is a constant $c>0$ such
that $\lambda_1(\Gamma_N)\geq c$ for each $N\in\mathbb{N}$. 
Since the boundary of $\Omega_{\Gamma_N}$ has $N$ boundary components
of length 1, Theorem~\ref{mainthm} leads to
\begin{gather*}
  \sigma_1(\Omega_{\Gamma_N})L(\partial\Omega_{\Gamma_N})\geq \alpha N
  \lambda_1(\Gamma_N)\geq c\alpha N.
\end{gather*}
This completes the proof of Theorem~\ref{Theorem:MainIntro}.

\section{Proof of the comparison results}
\label{Section:ProofComparison}

Let $f:\Omega_\Gamma\rightarrow\mathbb{R}$ be a smooth function. Given
a vertex $v\in V(\Gamma)$, the function $f_v$ is defined to be the
restriction of $f$ to the cylinder $C_v$. On 
each cylinder $C_v$, the function $f_v$ admits a decomposition
$f_v=\overline{f}_v+\tilde{f}_v$
where
$$\overline{f}_v(r)=\int_{\Sigma_v}f(r,x)\,dV_{\Sigma_v}(x)$$
is the average of $f$ on the corresponding slice of $C_v$. It follows
that for each $r\in[0,1]$,
$$\int_{\Sigma_v}\tilde{f}(r,x)\,dV_{\Sigma_v}(x)=0.$$
The function $\overline{f}$ is defined to be $\overline{f}_v$ on each
cylinder $C_v$, and similarly the function $\tilde{f}$ is defined to
be $\tilde{f}_v$ on each $C_v$.

Let $f\in C^\infty(\Omega_\Gamma)$ be a Steklov eigenfunction corresponding to
$\sigma_1(\Omega_\Gamma)$. The function
$$x=x_f:V(\Gamma)\rightarrow\mathbb{R}$$ 
is defined to be the average of
$f$ over the boundary component $\Sigma_v$. Since $|\Sigma_v|=1$ for
each vertex $v$, this is expressed by
\begin{gather*}
  x(v)=\int_{\Sigma_v}f\,dV_{\Sigma_v}=\overline{f}_v(0).
\end{gather*}
Because
\begin{gather*}
  \sum_{v\in V(\Gamma)}x(v)=\sum_{v}\int_{\Sigma_v}f\,dV_{\Sigma_v}=\int_{\Sigma}f\,dV_\Sigma=0,
\end{gather*}
the function $x$ can be used as a trial function in the
variational characterization~(\ref{charvargraph}) of
$\lambda_1(\Gamma)$. 
It follows from the orthogonality of $\overline{f}$ and $\tilde{f}$ on
the boundary $\Sigma=\partial\Omega_\Gamma$ that
\begin{align}\label{Inequality:DebutDePreuve}
  \int_{\partial\Omega_\Gamma}f^2\,dV_\Sigma
  &=\int_{\partial\Omega_\Gamma}(\overline{f}+\tilde{f})^2\,dV_\Sigma\nonumber\\
  &=\sum_{v\in V(\Gamma)}x(v)^2+
  \int_{\partial\Omega_\Gamma}\tilde{f}^2\,dV_\Sigma
  \leq \frac{1}{\lambda_1(\Gamma)}q_\Gamma(x)+
  \int_{\partial\Omega_\Gamma}\tilde{f}^2\,dV_\Sigma.
\end{align}
The two terms on the right-hand side of the previous inequality will
be bounded above in terms of $\|\nabla f\|^2_{L^2(\Omega_\Gamma)}$. In
order to bound $\int_{\partial\Omega_\Gamma}\tilde{f}^2\,dV_\Sigma$,
it will be sufficient to consider the behaviour of $\tilde{f}$ locally
on each cylinders $C_v$. More work will be required to bound
$q_\Gamma(x)$.

\subsection{Local estimate of smooth functions on cylindrical neighbourhoods}
On the model cylinder $C_0=[0,1]\times\Sigma_0$, consider the
following mixed Neumann--Steklov spectral problem:
\begin{gather}\label{SloshingProblem}
  \Delta f=0\mbox{ in }(0,1)\times{\Sigma_0},\\
  \partial_nf=0\mbox{ on }\{1\}\times{\Sigma_0},\nonumber\\
  \partial_nf=\mu f\mbox{ on }\{0\}\times{\Sigma_0}.\nonumber
\end{gather}
This problem is related to the sloshing spectral problem.
See~\cite{BanKulcPolt,Moiseev} for details.

\begin{lemma}\label{Lemma:localestimate}
  Let $\mu$ be the first non-zero eigenvalue of the sloshing problem~(\ref{SloshingProblem}).
  For any smooth function $f:\Omega_\Gamma\rightarrow\mathbb{R}$,
  \begin{gather}\label{Inequality:Localestimate}
    \int_{\partial\Omega_\Gamma}\tilde{f}^2\,dV_\Sigma\leq
    \mu^{-1}\int_{\Omega_\Gamma} |\nabla f|^2.
  \end{gather}
\end{lemma}
\begin{proof}
The first non-zero eigenvalue of this problem is characterized by
\begin{gather}\label{VarCharSloshing}
  \mu=\inf\left\{ \frac{\int_{(0,1)\times\Sigma_0} |\nabla f |^2}{\int_{\{0\}\times\Sigma_0}f^2}\,:\, f\in C^\infty([0,1]\times\Sigma_0),
    \int_{\{0\}\times\Sigma_0}f\,ds=0\right\}.
\end{gather}
Since $\tilde{f}$ is orthogonal to constants on each boundary
component $\Sigma_v$, it follows from~(\ref{VarCharSloshing}) that
$$
\int_{\partial\Omega_\Gamma}\tilde{f}^2\,dV_\Sigma\leq
\mu^{-1}\sum_{v\in V(\Gamma)}\int_{C_v} |\nabla\tilde{f}|^2.
$$
The proof of Lemma~\ref{Lemma:localestimate} is completed by observing that
the Dirichlet energy of $f_v:C_v\rightarrow\mathbb{R}$
is expressed by
\begin{gather*}
  \int_{C_v}|\nabla f_v|^2=\int_0^1\left(\overline{f_v}(r)'\right)^2\,dr+\int_{C_v}|\nabla\tilde{f}_v|^2.
\end{gather*}

\end{proof}

\subsection{Global estimate and graph structure}

\begin{lemma} \label{Lemma:GlobalEstimate}
  There exists a positive constant $C_0$ depending only on the fundamental
  piece $M_0$ such that the following holds for any  function $f$ on 
  $\Omega_\Gamma$
  \begin{equation}
     \sum_{v \sim w}\bigl(x(v)-x(w)\bigr)^2\leq C_0 \int_{\Omega_\Gamma} |\nabla f|^2.
  \end{equation}
\end{lemma}

The proof of Lemma~\ref{Lemma:GlobalEstimate} is based on the
following general estimate.
\begin{lemma}\label{Lemma:EstimateAverage}
  Let $\Omega$ be a smooth compact connected Riemannian surface with boundary. Let $A$ and
  $B$ be two of the connected components of the boundary
  $\partial\Omega$, both of length 1.
  There exists a constant $C>0$ depending only on $\Omega$ such that 
  any smooth function $f\in C^\infty(\Omega)$ satisfies
  \begin{gather}
    \left|\int_{A}f-\int_{B}f\right|^2\leq C\int_\Omega|\nabla f|^2
  \end{gather}
\end{lemma}
In fact, we will use this estimate only for harmonic functions, in
which case it is also possible to prove it using a method similar to
that of~\cite{colbmat}.
\begin{proof}[Proof of Lemma~\ref{Lemma:EstimateAverage}]
Let $x=\int_A f$, $y=\int_Bf$ be the average of $f$ on the two
boundary components $A$, $B$. Let
$$\langle f\rangle=\frac{1}{|\Omega|}\int_\Omega f,\quad
$$
be the average of $f$ on the surface $\Omega$. Finally, set
$g=f-\langle f\rangle.$
Now, since $\int_\Omega g=\int_\Omega (f-\langle f\rangle)=0$,
\begin{gather*}
  \int_\Omega g^2\leq\mu^{-1}\int_\Omega|\nabla g|^2=\mu^{-1}\int_\Omega|\nabla f|^2,
\end{gather*}
where $\mu>0$ is the first non zero Neumann eigenvalue of $\Omega$.
It follows that
\begin{gather}
  \label{ineqproofglobal}
  \|f-\langle f\rangle\|_{H^1(\Omega)}^2=\|g\|_{H^1(\Omega)}^2=\|g\|_{L^2(\Omega)}^2+\|\nabla f\|_{L^2(\Omega)}^2\leq
  (\mu^{-1}+1)\|\nabla f\|_{L^2(\Omega)}^2.
\end{gather}
In other words, the Dirichlet energy of $f$ controls how far $f$ is
from its average $\langle f\rangle$ in $H^1$-norm. This is essentially
a version of the Poincaré Inequality. The restriction of $f$ to $A$ and $B$
are also close to the average $\langle f\rangle$ in $L^2$-norm. Indeed, it follows from
the fact that the trace operators $\tau_a:H^1(\Omega)\rightarrow
L^2(A)$ and $\tau_B:H^1(\Omega)\rightarrow L^2(B)$ are bounded that
\begin{gather*}
  |x-\langle f\rangle|=|\int_A(f-\langle
  f\rangle)|=|\int_A(\tau_A(g))|
  \leq\|\tau_A(g)\|_{L^2(A)}
  \leq
  \|\tau_A\|\|g\|_{H^1(\Omega)},
\end{gather*}
and similarly $|y-\langle f\rangle|\leq\|\tau_B\|\|g\|_{H^1(\Omega)}$,
where $\|\tau_A\|$ and $\|\tau_B\|$ are the operator norms.
These two inequalities together lead to
\begin{align*}
  |x-y|\leq|x-\langle f\rangle|+|y-\langle f\rangle|
  \leq (\|\tau_A\|+\|\tau_B\|)\|g\|_{H^1(\Omega)}.
\end{align*}
In combination with (\ref{ineqproofglobal}) this imply
\begin{align*}
  |x-y|^2\leq (\|\tau_A\|+\|\tau_B\|)^2(\mu^{-1}+1))\int_\Omega|\nabla f|^2.
\end{align*}
One can take $C=(\|\tau_A\|+\|\tau_B\|)^2(\mu^{-1}+1))$. The proof is completed.
\end{proof}

\begin{proof}[Proof of Lemma~\ref{Lemma:GlobalEstimate}]
  For each adjacent vertices $v\sim w$ of the graph $\Gamma$, we apply
  Lemma~\ref{Lemma:EstimateAverage} to the surface $M_v\cup M_w$ with $A=\Sigma_v$ and
  $B=\Sigma_w$ to get
  $$\bigl(x(v)-x(w)\bigr)^2\leq C\int_{M_v\cup M_w}|\nabla f|^2.$$
  Since the graph $\Gamma$ is regular of degree $k$, it follows that
  $$
  \sum_{v\sim w} \bigl(x(v)-x(w)\bigr)^2\leq C\sum_{v\sim w}\int_{M_v\cup M_w}|\nabla f|^2
  =Ck\int_{\Omega_\Gamma}|\nabla f|^2
  $$
\end{proof}

% \begin{figure}
%   \centering
%   \psfrag{cv}[][][1]{$C_v$}
%   \psfrag{cw}[][][1]{$C_w$}
%   \psfrag{mv}[][][1]{$M_v$}
%   \psfrag{mw}[][][1]{$M_w$}
%   \psfrag{g}[][][1]{$\gamma$}
%   \includegraphics[width=7cm]{Figures/twopieces.eps}
%   \caption{The path $\gamma$ joins $p_v\in C_v$ to $p_w\in C_w$.}
%   \label{twopieces}
% \end{figure}

\subsection{The proof of Theorem~\ref{mainthm}}
\subsubsection*{The upper bound}
Let $f$ be a Steklov eigenfunction corresponding to the first non-zero
Steklov eigenvalue $\sigma_1(\Omega_\Gamma)$.
Combining the local estimate obtained in
Lemma~\ref{Lemma:localestimate} and the global estimate of Lemma~\ref{Lemma:GlobalEstimate} with
Inequality~(\ref{Inequality:DebutDePreuve}) leads to
\begin{align*}
  \int_{\partial\Omega_\Gamma}f^2\,dV_\Sigma
  \leq
  \left(\frac{C_0}{\lambda_1(\Gamma)}+\frac{1}{\mu}\right)\int_{\Omega_\Gamma}|\nabla
  f|^2,
\end{align*}
 which of course can be rewritten
 \begin{gather*}
   \sigma_1(\Omega)=
   \frac{\int_{\Omega_\Gamma}|\nabla
     f|^2}{\int_{\partial\Omega_\Gamma}f^2}
   \geq\left(\frac{C_0}{\lambda_1(\Gamma)}+\frac{1}{\mu}\right)^{-1} 
   =\frac{\lambda_1(\Gamma)}{\lambda_1(\Gamma)/\mu+C_0}
   \geq\min\left\{\mu,\frac{1}{C_0}\right\}\frac{\lambda_1(\Gamma)}{\lambda_1(\Gamma)+1}
\end{gather*}
Now, because we are on a regular graph of degree $k$, $\lambda_1\leq
k$, so that
$$\sigma_1(\Omega)\geq \frac{1}{k+1}\min\left\{\mu,\frac{1}{C_0}\right\}\lambda_1(\Gamma),$$
and taking $\beta=\frac{1}{k+1}\min\left\{\mu,\frac{1}{C_0}\right\}$,
the proof of Theorem~\ref{mainthm} is completed.

\subsubsection*{The lower bound}
Let $x:\in\ell^2(\Gamma)$ be an normalized eigenfunction corresponding to
$\lambda_1(\Gamma)$. The function $x$ satisfy
$$\sum_{v\in V(\Gamma)}x(v)^2=1\quad\mbox{ and }\quad
\sum_{v\in V(\Gamma)}x(v)=0.$$
Using $x$, a function $f_x:\Omega_\Gamma\rightarrow\mathbb{R}$ is
defined to be $x(v)$ on $\Sigma_v$ and to decay linearly to zero on
the cylinder $C_v$. The function $f_x$ satisfy
$$\int_{\partial\Omega_\Gamma}f_x=\sum_{v\in V}x(v)=0,$$
and can therefore be used in the variational characterization of $\sigma_1(\Omega_\Gamma)$.
The estimates of the Rayleigh quotient are simple and
follows~\cite[p. 290]{colbmat} verbatim. We will not reproduce it
here.

\subsection*{Acknowledgements}
The authors are thankful to Talia Fern\'os, Iosif Polterovich, and
Thomas Ransford for useful conversations.
\bibliographystyle{plain}
\bibliography{biblioCG}

\end{document}